\documentclass[oneside,notitlepage,12pt]{article}

\usepackage{amssymb}
\usepackage[leqno]{amsmath}
\usepackage{amsfonts}
\usepackage{amsopn}
\usepackage{amstext}
\usepackage{amsthm}

\usepackage[colorlinks]{hyperref}
\providecommand{\cal}{\mathcal}
\renewcommand{\Bbb}{\mathbb}

\newenvironment{pf}{\begin{proof}}{\end{proof}}

\newcommand{\Aaa}{{\cal{A}}}
\newcommand{\Bee}{{\cal{B}}}
\newcommand{\Cee}{{\cal{C}}}
\newcommand{\Dee}{{\cal{D}}}
\newcommand{\Ef}{{\cal{F}}}

\newcommand{\Pee}{{\cal{P}}}

\newcommand{\Nat}{{\Bbb{N}}}

\newcommand{\Err}{{\Bbb{R}}}

\newcommand{\al}{\alpha}

\renewcommand{\phi}{\varphi}
\renewcommand{\rho}{\varrho}

\newcommand{\rest}{\restriction}

\newcommand{\ntr}{{n\in\omega}}

\newcommand{\loe}{\leqslant}
\newcommand{\goe}{\geqslant}

\newcommand{\subs}{\subseteq}

\newcommand{\nnempty}{\ne\emptyset}

\newcommand{\ovr}{\overline}

\newcommand{\cl}{\operatorname{cl}}

\newcommand{\cf}{\operatorname{cf}}

\newcommand{\supp}{\operatorname{supp}}

\newtheorem{tw}{Theorem}[section]

\newtheorem{lm}[tw]{Lemma}
\newtheorem{prop}[tw]{Proposition}

\theoremstyle{definition}
\newtheorem{df}[tw]{Definition}
\newtheorem{ex}[tw]{Example}

\newtheorem{question}[tw]{Question}
\theoremstyle{remark}

\newcommand{\setof}[2]{\{#1\colon #2\}}

\newcommand{\sett}[2]{\{#1\}_{#2}}
\newcommand{\sn}[1]{\{#1\}} % singleton
\newcommand{\dn}[2]{\{#1,#2\}} % doubleton
\newcommand{\pair}[2]{\langle #1, #2 \rangle} % pair
\newcommand{\map}[3]{#1\colon #2 \to #3} % A function
\newcommand{\im}{\operatorname{im}}
\newcommand{\dpower}[2]{[#1]^{#2}}

\newcommand{\ciag}[1]{{\sett{{#1}_n}{\ntr}}}

\newcommand{\iso}{\approx}

\newcommand{\norm}[1]{\|#1\|}

\newcommand{\uball}[1]{\ovr B_{#1}}
\newcommand{\ubal}{\uball}

\newcommand{\cmp}{\circ} % composition!!!

\title{Almost disjoint families of countable sets and separable complementation properties}
\author{
{Jes\'us Ferrer}
\footnote{The first author has been partially supported by 
MEC and FEDER Project MTM2011-22417.}
\\
{\small Department of Mathematical Analysis, University of Valencia}\\
{\small Dr. Moliner 50, 46100 Burjassot (Valencia), Spain}\\
{\small\texttt{Jesus.Ferrer@uv.es}}
\and
{Piotr Koszmider}
\footnote{The second author has been partially supported by the National Science Center research grant 2011/01/B/ST1/00657.}
\\
{\small Mathematical Institute, Polish Academy of Sciences}\\
{\small ul. \'Sniadeckich 8, 00-956 Warszawa, Poland}\\
{\small\texttt{P.Koszmider@Impan.pl}}
\and
{Wies{\l}aw Kubi\'s}
\footnote{The third author has been supported by the GA\v CR grant P 201/12/0290. }
\\
{\small Institute of Mathematics, Academy of Sciences of the Czech Republic}\\
{\small \v Zitn\'a 25, 115 67 Praha 1, Czech Republic}\\
{\small\texttt{kubis@math.cas.cz}}
}
\newcommand{\el}{l}
\newcommand{\aue}[1]{K_{\pair{#1}\omega}}

\begin{document}

\maketitle

\begin{abstract}
We study the separable complementation property (SCP) and its natural variations in Banach spaces of continuous functions over compacta $K_\Aaa$ induced by almost disjoint families $\Aaa$ of countable subsets of uncountable sets. 
For these spaces, we prove among others that $C(K_\Aaa)$ has the controlled variant of the separable complementation property if and only if $C(K_\Aaa)$ is Lindel\"of in the weak topology if and only if $K_\Aaa$ is monolithic.
We give an example of $\Aaa$ for which $C(K_\Aaa)$ has the SCP, while $K_\Aaa$ is not monolithic and an example of a space $C(K_\Aaa)$ with controlled and continuous SCP which has neither a projectional skeleton nor a projectional resolution of the identity.
Finally, we describe the structure of
almost disjoint families of cardinality $\omega_1$ which induce monolithic spaces of the form $K_\Aaa$: 
They can be obtained from countably many ladder systems and pairwise disjoint families
applying  simple operations.


\end{abstract}

\tableofcontents

\section{Introduction}

In this paper we deal with  Banach spaces of real-valued continuous functions 
defined on a compact Hausdorff space $K$ with the supremum norm and denoted $C(K)$.
Actually, the compact spaces we consider are induced by a combinatorial object---an almost disjoint
family. Recall that a family of infinite sets
 $\Aaa$ is \emph{almost disjoint} if $A\cap B$ is finite whenever $A, B\in \Aaa$ are different.
Additionally, we shall only consider almost disjoint families of countable sets. Almost disjoint families of $\Nat$ have found
many applications in the theory of Banach spaces, some classical cases include
Whitley's short  proof of Philips' theorem in \cite{whitley}, Haydon's
constructions of Grothendieck spaces in \cite{haydon}, 
or Johnson and Lindenstrauss' counterexamples concerning 
weakly compactly generated Banach spaces in \cite{johnsonlind}. We, however, will focus on
almost disjoint families of countable subsets of uncountable sets. Several of such families
induce Banach spaces which
could also be considered classical (counter)-examples like the Ciesielski-Pol spaces
(see \cite{ciesielskipol, DGZ}) or continuous functions on the ladder system spaces (\cite{pollindelof},
\cite{ciesielskipol}).

Let us be more precise on how we obtain a compact space and then the Banach space
from an almost disjoint family.
Fix  a family $\Aaa$ of countable almost disjoint sets 
with $X = \bigcup\Aaa$. We define a compact space $K_\Aaa$ as follows.
The points of $K_\Aaa$ are
$$X \cup \setof{\el_A}{A\in\Aaa} \cup \sn \infty,$$
where each point of $X$ is declared to be isolated, a basic neighborhood of $\el_A$ is of the form $\sn{\el_A}\cup (A\setminus F)$ where $F$ is finite, and a basic neighborhood of $\infty$ is of the form $$K_\Aaa \setminus \bigcup_{i < n}(\sn{\el_{A_i}}\cup A_i)$$
where $\setof{A_i}{i < n} \subs \Aaa$.
The space $K_\Aaa$ is compact scattered of height 3 with $\infty$ being the unique point at the maximal level.
Furthermore, $K_\Aaa \setminus \sn\infty$ is first countable iff $\Aaa$ consists of countable sets.
In order to shorten some expressions, given $\Bee \subs \Aaa$, we shall write $\el_\Bee$ for the set $\setof{\el_B}{B\in \Bee}$. For countable $X$ these spaces can be traced to Alexandroff and Urysohn, and later were known as Mr\'owka's
spaces, $\Psi$-spaces or Isbell spaces. One could also note that $K_\Aaa$ is the Stone space of
the Boolean algebra generated by $\Aaa$ and the finite subsets of $\bigcup\Aaa$.

We investigate the Banach spaces of the form $C(K_\Aaa)$ focusing on how some aspects of their geometry depend
 on the topological properties of $K_\Aaa$
which in turn depend on the combinatorial properties of $\Aaa$. To be more precise, we investigate the family of 
complemented separable subspaces of $C(K_\Aaa)$ depending on $\Aaa$. The links between
almost disjoint families and projections can be traced back to \cite{whitley}. 
Recall that a subspace $F$ of a Banach space $X$ is \emph{complemented} if there exists a bounded linear operator (called \emph{projection}) $\map P X X$ satisfying $P^2 = P$ and $\im P = F$, where $\im P$ denotes the image of $P$.
We consider several properties of Banach spaces
in terms of the structure  of some families of  projections on separable subspaces
ranging from the existence of unconditional basis to the existence of projectional skeletons. Let us
recall some of these properties here:

\begin{df}\label{maindefinition} Let $X$ be a Banach space.
\begin{enumerate}
\item $X$ has the \emph{separable complementation property} (SCP) if  every
separable subspace of $X$ is contained in a separable complemented subspace.
\item $X$ has the \emph{controlled separable property} (briefly: \emph{controlled SCP}) if for every separable subspace $F \subs X$ and for every separable subspace $G \subs X^*$ there exists a projection $\map P X X$ such that $F \subs \im P$, $G\subs \im P^*$ and $\im P$ is separable.
\item $X$ has the \emph{continuous separable complementation property} (briefly: \emph{continuous SCP}) if there
is a family $\Ef$ of separable complemented subspaces
of $X$ such that $X = \bigcup\Ef$, $\Ef$ is up-directed and for every sequence $F_0 \subs F_1 \subs \dots$ in $\Ef$ the closure of the union $\bigcup_{\ntr}F_n$ is in $\Ef$. A family $\Ef$ as above is called a skeleton of
complemented subspaces of $X$.
\item $X$ has a \emph{projectional skeleton} if it has a skeleton as above such that there exist  projections
$P_E: X\rightarrow E$ which satisfy $P_E\circ P_F=P_F\circ P_E=P_E$ whenever $E\subseteq F$
with $E, F\in \Ef$.
\end{enumerate}
\end{df}

The SCP has appeared in the classical work of Amir and Lindenstrauss (\cite{AL}) where it is
proved that weakly compactly generated Banach spaces share this property and it motivated 
some main lines of research concerning non-separable Banach spaces. 
Probably the widest well-investigated, at the present moment, class of Banach spaces with this property
is the class of weakly Lindel\"of determined spaces (\cite{gonzalez}). It is closed under taking
quotients and subspaces, however SCP is not inherited by these operations. This follows from a
recent example of  Figiel, Johnson and Pe\l czy\'nski~(\cite{FJP}) for subspaces and
from a simple observation that $l_1(\kappa)$ has the SCP for any cardinal $\kappa$.
The terminology concerning the
controlled SCP and the continuous SCP is taken from  \cite{Woj} and  \cite{KalKub2}
respectively, however these properties were considered  in the literature earlier
(e.g. 4.4. \cite{plichkoyost}, 6, 7 \cite{Kub_lines});
controlled SCP has been recently in \cite{Fer, FerWoj};
projectional  skeletons were introduced and studied in \cite{Kub_skel}.
By definitions, the SCP is implied by the other properties, and the existence
of a projectional skeleton implies the continuous SCP. The remaining implications
are not valid even for $C(K)$ spaces for quite
particular classes. This follows  from our examples (\ref{nonmono}, \ref{nonprojskeleton}) as
well as from the main  examples of \cite{KalKub2, Kub_lines}.
One of the simplest examples of a $C(K)$ space with a projectional skeleton and without the controlled SCP comes by taking $K = \{0,1\}^{\omega_1}$.
It is well known that this $K$ belongs to the class of Valdivia compact spaces and therefore $C(K)$ has a projectional skeleton (see, e.g., \cite{KKLP, Kub_skel}).
On the other hand, $K$ is not monolithic, so $C(K)$ cannot have the controlled SCP.

Perhaps at the beginning of the discussion of the results of our
paper we should mention a result proved at the end of Section 2:

\begin{tw}\label{scpimpliescscp} Suppose that $\Aaa$ is an almost
disjoint family of countable sets. $C(K_\Aaa)$ has a continuous SCP if and only if it has the SCP.
\end{tw}

In Section 2 we analyze the simplest case of an almost disjoint family, namely, a {\it disjointifiable} family, in other words, a family that is equivalent (modulo finite sets) to a
pairwise disjoint one (see Definition \ref{equivdisjointdef}). It turns out that in this trivial case
corresponding to weakly compactly generated Banach space we have unconditional basis and hence
the richest family of commuting projections. We obtain:

\begin{tw}\label{eberleinlist} Suppose $\Aaa$ is an almost disjoint family of countable sets, then the following are equivalent:
\begin{enumerate}
\item $\Aaa$ is disjointifiable,
\item  $K_\Aaa$ is an  Eberlein compact,
\item $C(K_\Aaa)$ is isomorphic to $c_0(\kappa)$ for $\kappa=|\Aaa|+\omega$,
\item $C(K_\Aaa)$ has an unconditional basis,
\item $C(K_\Aaa)$ has a projectional skeleton.
\end{enumerate}
\end{tw}
\begin{proof}Apply \ref{eberleindisjoint}, \ref{eberleinc0}, \ref{eberleinandskeletons}.
\end{proof}

In Section 3 we consider the case corresponding to Banach spaces which are
Lindel\"of in the weak topology. We obtain the following:

\begin{tw}\label{monolithiclist} Suppose $\Aaa$ is an almost disjoint family of countable sets, then the following are equivalent
\begin{enumerate}
\item $\Aaa$ is locally small,
\item  $K_\Aaa$ is monolithic,
\item $C(K_\Aaa)$ has the controlled SCP,
\item $C(K_\Aaa)$ is Lindel\"of in the weak topology,
\item $B_{C(K_\Aaa)^*}$ is monolithic  in the weak$^*$ topology.
\end{enumerate}
\end{tw}
\begin{proof}Apply \ref{equivwlindelof} and \ref{equivmonolithic}.
\end{proof}

A locally small almost disjoint family is defined in \ref{locsmalldef},
a compact space $K$ is \emph{monolithic} if all of its separable subsets are second countable.
In case $K$ is scattered, this just means that separable subsets are countable.
The notion of a monolithic space is due to Arkhangel'skii (see~\cite{Ark_book}).
A monolithic space is often called \emph{$\omega$-monolithic}, since this notion clearly generalizes to higher cardinals.
Typical examples of monolithic spaces are \emph{Corson compacta}, namely, compact subsets of $\Sigma$-products $\Sigma(\Gamma)$, where $\Sigma(\Gamma)$ is the subspace of $\Err^\Gamma$ consisting of all functions $\map x \Gamma \Err$ whose support $\supp(x) = \setof{\al}{x(\al)\ne0}$ is countable.
It is obvious that the closure of every countable subset of $\Sigma(\Gamma)$ is compact and second countable (contained in $\Err^S$ for some countable $S \subs \Gamma$).

The above results and  our examples \ref{nonmono} and \ref{nonprojskeleton} 
decide completely which implications hold among the properties of Definition \ref{maindefinition}
namely (1)$\Leftrightarrow$(3), (4)$\Rightarrow $(3), (2)$\Rightarrow$(1) and (4)$\Rightarrow$(2).
The last implication follows from a well known fact that Eberlein compacta are monolithic, being Corson compact.

In general the controlled SCP does not imply the Lindel\"of property
in the weak topology. To see this, take any non-metrizable monolithic compact linearly ordered space $K$ (e.g., $K$ could be a compact Aronszajn line) and
apply Nakhmanson's
theorem \cite{nakhmanson}, which says that for a linearly ordered compact $K$ the Lindel\"of
property of $C(K)$ implies that $K$ is second countable; by~\cite[Thm. 1.6]{KalKub2}, $C(K)$ has the controlled SCP.
It is interesting that a sufficient condition given in~\cite[Thm. 12.33]{fabianetal} for the Lindel\"of property due to Orihuela has a flavor of the controlled SCP.

In Section 4 we obtain some structural results concerning locally small almost disjoint families of countable sets,
of cardinality $\omega_1$, to express them  we need the notion of a ladder system (see the beginning of Section 4)
and we need to introduce some simple definition: Suppose $\Aaa\cup\Bee$ is an almost disjoint family 
$\Aaa'\subseteq \Aaa$ and  $\phi: \Aaa'\rightarrow
\Bee$ is an injection. Then 
$$(\Aaa\setminus \Aaa')\cup\{A'\cup \phi(A'): A'\in \Aaa'\}\cup(\Bee\setminus \phi[\Aaa'])$$
 is denoted $\Aaa+_\phi\Bee$.
The main result of Section 4  is the following structural theorem:

\begin{tw}\label{structural} Suppose that $\Aaa$ is a locally small  almost disjoint family of countable
sets of cardinality $\omega_1$. Then $\Aaa$ is equivalent to the family of the form
$$\bigcup_{n\in \Nat}(\Aaa_n+_{\phi_n}\Bee_n)$$
where
all $\Aaa_n$s are ladder systems, all $\Bee_n$s are disjointifiable
and $\phi_n$s are injections defined  on subfamilies of $\Aaa_n$s.
\end{tw}

In other words, locally small almost disjoint families are obtained in a canonical way from countably many ladder systems and pairwise disjoint families. As noted at the end of Section 4, if all the families
$\Aaa_n$ are disjointifiable the $\Aaa$ is disjointifiable as well.
One should note that the structure of non locally small almost disjoint families
of sets, even in the case of  subsets of $\Nat$ is a much harder task with extensive literature devoted to it.

%{\bf How does it translate to the level of compact spaces or Banach spaces? In what sense compact spaces and Banach spaces $C(K_\Aaa)$ are made from ladders and Eberlein compact or ladder system spaces and copies of $c_0(\kappa)$?}.

Analyzing Banach spaces induced by the ladder system spaces in Section 4, we obtain the following

\begin{tw}\label{nonprojskeleton} There is an almost disjoint family $\Aaa$ such that 
$C(K_\Aaa)$  has controlled SCP and continuous SCP but it has neither a PRI nor a projectional skeleton.
\end{tw}

Actually the above result provides the first example of  a Banach space with the continuous SCP, not isomorphic to a space with a PRI. This question was raised in \cite{Kub_lines} and was motivated by
an observation that in a Banach space of density $\aleph_1$, a skeleton of $1$-complemented separable subspaces naturally leads to a PRI. 
We should  note here that in general the existence of a PRI in the spaces $C(K_\Aaa)$ does 
not imply even the SCP:

\begin{ex}
Let $\Bee \subs [\omega]^\omega$ be an almost disjoint family of size $\omega_1$ and let $\Aaa \subs [\omega_2 \setminus \omega]^\omega$ be a disjoint family of size $\omega_2$.
Let $\Cee = \Aaa \cup \Bee$.
The space $C(K_\Bee)$ is a complemented subspace of $C(K_\Cee)$ and the canonical copy of $c_0$ is not complemented in $C(K_\Bee)$ (see the proof of Proposition~\ref{uncomplemented} for details).
Thus, $C(K_\Cee)$ fails the SCP.
On the other hand, it is easy to construct a PRI on $C(K_\Cee)$.
In fact, let $\sett{P_\al}{\al < \omega_2}$ be a PRI on $C(K_\Aaa)$ and, knowing that $C(K_\Cee) = C(K_\Aaa) \oplus C(K_\Bee)$, extend
$P_{\omega_1}$ by enlarging its kernel to $\ker{P_{\omega_1}} \oplus C(K_\Bee)$. Finally, replace each $P_\al$ for $\al > \omega_1$ by the appropriate projection onto $\im P_\al \oplus C(K_\Bee)$.
\end{ex}

Having in mind the results of Section 3, the only property from \ref{maindefinition} which can hold for
non-monolithic $K_\Aaa$s is the SCP (and equivalent continuous SCP). It turns out that 
such an example exists, its construction is the subject matter of Section 6: 

\begin{tw}\label{nonmono} There is an almost disjoint family $\Aaa$ which is not locally small, i.e.,
none of the conditions of \ref{monolithiclist} like the controlled SCP are satisfied by $C(K_\Aaa)$, yet
the space $C(K_\Aaa)$  has the SCP.
\end{tw}

This is a different behavior than for linearly ordered compact spaces $K$, as it was proved in \cite{KalKub2}
that for these $C(K)$s the SCP is equivalent to controlled SCP and
to the monolithicity of $K$ (Theorem 1.6 of \cite{KalKub2}). Also this cannot happen if $K$ is separable
because then $\bigcup\Aaa$ is countable and we are in the realm of subsets of $\Nat$ (\ref{uncomplemented}).

In the different direction from our present research is the question how few complemented
separable subspaces could appear in the space of the form $C(K_\Aaa)$. Of course then we 
focus on spaces without SCP.
Some separable subspaces must be complemented as $K_\Aaa$s have nontrivial
convergent sequences because they are scattered, so we have many complemented copies of $c_0$.
In \cite{mrowka} the second author assuming the continuum
hypothesis constructed an almost disjoint family such that the only decompositions of $C(K_\Aaa)=E\oplus F$
are for $E\sim c_0$ and $F\sim C(K_\Aaa)$ or vice versa. This space does not have SCP
nor is weakly Lindel\"of.  

We shall use standard set-theoretic notation. For example, an ordinal $\al$ will be treated as the interval of ordinals $[0,\al)$. According to this notation, the set of natural numbers is $\omega$.
The symbol $\iso$ will denote the relation of
 being homeomorphic (when dealing with compact spaces) and  $\equiv$ and $\sim$ the relations of being isometric or isomorphic as Banach spaces, respectively.
$|X|$ denotes the cardinality of $X$. The symbol $\triangle$ denotes the symmetric difference, i.e., $A \triangle B = (A \setminus B) \cup (B \setminus A)$.
At some point we shall use Fodor's Pressing Down Lemma, so we state it below for the sake of completeness.

Let $\kappa$ be a regular cardinal.
Recall that a set $A \subs \kappa$ is \emph{closed} if $\sup C \in A$ whenever $C \subs A\cap \al$ for some $\al < \kappa$ and $A$ is \emph{unbounded} if $\sup A = \kappa$.
A closed and unbounded set is called briefly a \emph{club}.
A set $S$ is \emph{stationary} in $\kappa$ if $S \cap A \nnempty$ whenever $A \subs \kappa$ is closed and unbounded.
All these notions, as well as the proof of Fodor's Lemma, can be found, e.g., in \cite{jech}.

\begin{lm}[Fodor's Pressing Down Lemma]
Let $\kappa$ be an uncountable regular cardinal and let $S \subs \kappa$ be a stationary set.
Let $\map f S \kappa$ be a regressive function, that is, $f(\xi) < \xi$ for every $\xi \in S$.
Then there exists a stationary set $T \subs S$ such that $f \rest T$ is constant.
\end{lm}

We shall use Fodor's Lemma for $\kappa = \omega_1$ only.
Roughly speaking, when dealing with a family of (usually countable) structures living in $\omega_1$ and indexed by countable ordinals, the Pressing Down Lemma allows us to refine the family (still having an uncountable one) fixing finitely many parameters defined as regressive functions.

\section{Weakly compactly generated spaces $C(K_\Aaa)$}

\begin{df}\label{equivdisjointdef} We say that two almost disjoint families of countable sets $\Aaa$ and $\Bee$ are \emph{equivalent} if
the Boolean algebra generated by $\Aaa$ and finite subsets of $\bigcup \Aaa$
is isomorphic to the Boolean algebra generated by $\Bee$ and the finite subsets of $\bigcup\Bee$.
We say that an almost disjoint family of countable sets is \emph{disjointifiable} if
$\Aaa$ is equivalent to a pairwise disjoint family.
\end{df}

Note that if $i$ is a Boolean isomorphism as in the definition above, then $i$ is 
a bijection between $\Aaa$ and $\Bee$ as well as $i$ induces a bijection 
between $\bigcup\Aaa$ and $\bigcup\Bee$.

\begin{lm}\label{lemmadisjoint} Suppose that $\Aaa$ is an almost disjoint family of countably infinite sets.
\begin{enumerate}
	\item[(1)]  $\Aaa$ is disjointifiable if and only if
for every $A\in \Aaa$ there is a finite set $F_A$ such that $\{A\setminus F_A: A\in \Aaa\}$
is a pairwise disjoint family.
	\item[(2)] $\Aaa$ and $\Bee$ are equivalent, if and only if $K_\Aaa$ and $K_\Bee$ are homeomorphic.
	\item[(3)] If $\Aaa$ is point countable, then it is disjointifiable.
	\item[(4)] Suppose $\Aaa=\bigcup_{n\in\Nat} \Aaa_n$ is an almost disjoint family and
each subfamily $\Aaa_n$ is disjointifiable, then $\Aaa$ is disjointifiable.
\end{enumerate}
\end{lm}
\begin{pf}
(1) Suppose $\map i {K_\Aaa}{K_\Bee}$ is a homeomorphism.
Then $i$ induces a bijection from $\Aaa$ onto $\Bee$ so that $i(A) \triangle A$ is finite for every $A \in \Aaa$.
Thus, it is obvious that a disjointifiable family can be ``corrected" to a disjoint family by subtracting finite sets.
Now suppose $\Aaa = \sett{A_\al}{\al < \kappa}$ is such that $\Bee = \sett{A \setminus F_\al}{\al < \kappa}$ is a disjoint family, where $F_\al \subs A_\al$ is finite for each $\al < \kappa$.
By induction, we can decrease each $F_\al$ so that the family remains disjoint. By this way, we shall get $\bigcup \Bee = \bigcup \Aaa$, that is, the spaces $K_\Aaa$ and $K_\Bee$ have the same sets of isolated points.
There is a unique continuous map $\map i {K_\Aaa}{K_\Bee}$ which is identity on the set of isolated points.
It is clear that $i$ is a homeomorphism.

(2) This is the Stone duality.

(3) For $A, A'\in\Aaa$ we say that $A\sim A'$ if and only if there are $A=A_0, ..., A_n=A'$
such that $A_i\cap A_{i+1}\not=\emptyset$ for each $0\leq i<n$. The point countability of
$\Aaa$ implies that the equivalence classes of $\sim$ are countable and have disjoint unions.
As countable families are disjointifiable, we conclude that $\Aaa$ is disjointifiable.

(4) We may assume that the families $\Aaa_n$ do not have elements in common.
Each of $\Aaa_n$'s is equivalent to a pairwise disjoint family $\Bee_n$, so
$\bigcup_{n\in\Nat} \Bee_n$ is point countable, and so by (3) it is equivalent
to a pairwise disjoint family. It is clear that $\bigcup_{n\in\Nat}\Bee_n$ is equivalent to $\Aaa_n$
by (1).

\end{pf}

\begin{prop}\label{eberleindisjoint} $K_\Aaa$ is an Eberlein compact if and only if
$\Aaa$ is disjointifiable.
\end{prop}
\begin{proof}
We will use Rosenthal's characterization of Eberlein compacta which says that
$K$ is an Eberlein compactum if and only if it has a point separating $\sigma$-point
finite family of open $F_\sigma$ sets (see \cite{Ros}, \cite[Thm. 19.30]{KKLP} or \cite[Thm. 12.13]{fabianetal}). 
It is clear that in the case of a totally disconnected space we
may assume that the $F_\sigma$ sets are clopen.

If $\Aaa$ is disjointifiable, then  $K_\Aaa$ is homeomorphic to $K_\Bee$
for $\Bee$ pairwise disjoint. Let $\phi_B:\Nat\rightarrow [B]^{<\omega}$ be an enumeration
of all finite subsets of $B$ for $B\in\Bee$. If we define $\Bee_n=\{B\setminus\phi(n): B\in\Bee\}$
we obtain another pairwise disjoint family. These families together with the family of
all one element subsets of $\bigcup\Bee$ form the required point separating $\sigma$-point
finite family of clopen  sets.

Now suppose that $\mathcal F_n$ are point separating $\sigma$-point
finite family of clopen sets of $K_\Aaa$. Only countably many of these sets may contain
$\infty$ and those sets miss only finitely many points of the form $l_A$ for $A\in\Aaa$.
So by \ref{lemmadisjoint}(3) it is enough to use \ref{lemmadisjoint}(4) to conclude that $\Aaa$ is disjointifiable.
\end{proof}

Assume $\Aaa$ is a disjoint family of countably infinite sets and $|\Aaa| = \kappa \goe \omega$.
Then $K_\Aaa$ is naturally homeomorphic to
$$\aue \kappa := \sn \emptyset \cup \dpower \kappa 1 \cup \setof{\dn \al {\kappa + n}}{\al \in \kappa, \ntr},$$
endowed with the Cantor cube topology inherited from $\wp(\kappa) \iso 2^\kappa$.
Thus, given an infinite cardinal $\kappa$, there exists a unique topological type of a compact space $K_\Aaa$ coming from an almost disjoint family of countable sets such that $|\Aaa| = \kappa$.

\begin{prop}\label{eberleinc0} If $\Aaa$ is an almost disjoint family of countable sets such that $K_\Aaa$ is Eberlein compact, then
$C(K_\Aaa)\sim c_0(\kappa)$ for $\kappa=|\Aaa|+\omega$, in particular
$C(K_\Aaa)$ has an unconditional basis.
\end{prop}

\begin{proof} We may assume that $\Aaa$ is pairwise disjoint. Consider 
$$X=\{f\in C(K_\Aaa): f(l_A)=0=f(\infty)\  \hbox{for each}\ A\in\Aaa\}.$$
Define $P:C(K_\Aaa)\rightarrow X$ by $P(f)(x)=f(x)-f(l_A)$ where
$x\in \bigcup\Aaa$ and $A$ is the unique element of $\Aaa$ such that $x\in A$.
It is easy to check that $P$ is a projection onto $X$. 
But $X\equiv c_0(\bigcup\Aaa)$ and $C(K_\Aaa/X)\sim c_0(|\Aaa|)\oplus\Err$
which completes the proof of the proposition. 
\end{proof}

\begin{lm}\label{nocopyw1}
Suppose that $\Aaa$ is an almost disjoint family of countable sets, then
$[0,\omega_1]$ is not  a subspace of the dual ball $B_{C(K_\Aaa)}$ with the weak$^*$
topology.
\end{lm}
\begin{proof} Suppose that $\phi: [0,\omega_1]\rightarrow B_{C(K_\Aaa)}$ is a homeomorphic embedding.
 As
convergence in the dual norm implies the weak$^*$ convergence, we may cover
the copy of $[0,\omega_1)$ by norm-open sets $U_\alpha$ such that 
$$\phi(\alpha)\in
\phi[[0,\omega_1)]
\cap U_\alpha\subseteq \phi[[0,\alpha]].$$
Let $\mathcal V$ be a locally point-finite refinement of $U_\alpha$s
which exists as metrizable spaces are paracompact (\cite{engelking}).
Whenever there is $\beta<\alpha$ and $V\ni \phi(\alpha)$ such that $\phi(\beta)\in V\in \mathcal V$
let $f(\alpha)$ be one of them. If the set of such $\alpha$s were stationary, by the
Pressing Down Lemma, $f$ would be constant on an uncountable set which would contradict
the point-finiteness of $\mathcal V$. Hence, for a club set of $\alpha$s there is no such $\beta$
and so such $\phi(\alpha)$ can be separated by a norm open set from
$\phi(\beta)$s for $\beta<\alpha$. $\phi(\alpha)$ can also be separated by
a norm open set from the $\phi(\beta)$s for $\beta>\alpha$. As club subsets of $\omega_1$
are homeomorphic to $\omega_1$, we may assume that $\phi([0,\omega_1))$ is norm discrete.
By the fact
that the ideal of non-stationary subsets of $\omega_1$ is $\sigma$-complete (see \cite{jech}), we can conclude that
there is a stationary set $S\subseteq \omega_1$ such that there is $\delta>0$ such that
for every distinct $\alpha, \beta\in S$ we have $||\phi(\alpha)-\phi(\beta)||>\delta$, where the norm is
the dual norm in $C(K_\Aaa)^*$. It is well known that stationary sets contain closed subsets
of arbitrarily big countable order type. This implies that for every countable
ordinal $\alpha$ the ball $B_{C(K_\Aaa)}$ with the weak$^*$ topology
 contains homeomorphic copies $K_\alpha$ of ordinal intervals $[0,\alpha]$ such that the distance between
any two points is not smaller than $\delta$. 
Given $\alpha<\omega_1$, as $[0,\alpha]$ is countable and supports
of the measures corresponding to the elements of the dual space are countable,
there is a countable $\Bee\subseteq \Aaa$
such that  taking the restriction of elements of $K_\alpha$ to 
$$X(\Bee)=\{1_{x}: x\in \bigcup\Bee\}\cup\{1_{A\cup\{l_A\}}: A\in \Bee\},$$
gives an injective weak$^*$ continuous function, and so a homeomorphism onto its image. Moreover, we may assume that the norms
of elements of $K_\alpha$ do not change when we take the restrictions.
This leads to a contradiction as $X(\Bee)$ is isomorphic to $c_0$ by Proposition~\ref{eberleinc0}
and so the Szlenk index of it is $\omega$ (\cite[Prop. 2.27]{rosenthal}).
By~\cite[Cor. 2.20]{rosenthal}, the Szlenk index must be bigger than the Cantor-Bendixson 
height of all countable ordinal intervals, once we have the above $\delta$-separated copies of them in the dual ball of $X(\Bee)$.
This completes the proof of the lemma.
\end{proof}

\begin{prop}\label{eberleinandskeletons} If 
$C(K_\Aaa)$ has  a projectional skeleton, then $K_\Aaa$ is an Eberlein compact.
\end{prop}
\begin{proof}
Assume $C(K_\Aaa)$ has a projectional skeleton $\Pee$.
There is a renorming of $C(K_\Aaa)$ for which $\Pee$ becomes a $1$-projectional skeleton (see \cite[Cor. 16]{Kub_skel}).
Let $G$ be the dual unit ball with respect to this renorming.
Then $G$ belongs to the class $\mathcal R$ introduced in \cite{retractions}.
By \cite[Thm. 4.3]{retractions} a member of class $\mathcal R$ is either a Corson compact or
contains a copy of $[0,\omega_1]$ which is excluded by Lemma~\ref{nocopyw1}, because $G$ is contained in a multiple of the standard dual ball.
Since $K_\Aaa$ is contained in a multiple of $G$, it follows that $K_\Aaa$ is Corson.
But $K_\Aaa$ is scattered, and hence
Eberlein by Alster's theorem \cite{alster} which says that a scattered Corson compact space is Eberlein compact.
\end{proof}

\begin{proof}[Proof of Theorem~\ref{scpimpliescscp}]
First, recall Sobczyk's Theorem~\cite{Sobczyk}: Every isomorphic copy of $c_0$ is complemented in every separable Banach space.
It follows that, if $C(K_\Aaa)$ has SCP, then
every isomorphic copy of $c_0$ in $C(K_\Aaa)$ is complemented. 
Consider the subspaces $X(\Bee)$ which are closures of linear spaces generated by
constant functions and by
$$\{1_{x}: x\in \bigcup\Bee\}\cup\{1_{A\cup\{l_A\}}: A\in \Bee\},$$
where $\Bee$ is a countable subset of $\Aaa$. 
It is enough to note that 
$$\{X(\Bee): \Bee\ \hbox{is a countable subset of}\ \Aaa\}$$
forms a skeleton of separable complemented subspaces. They are complemented
being isomorphic to spaces $C(K_\Bee)$ respectively, and so isomorphic to
$c_0$ by Proposition~\ref{eberleinc0} since $K_\Bee$ is metrizable, and so Eberlein compact.
\end{proof}

\section{Weakly Lindel\"of spaces of the form $C(K_\Aaa)$}

\begin{df}\label{locsmalldef}  We say that an almost disjoint family of countable sets $\Aaa$ is
\emph{locally small} if 
$$\{A\in\Aaa: A\cap B\ \hbox{is infinite}\}$$
is at most countable for each countable $B\subseteq\bigcup\Aaa$.
\end{df}

\begin{prop}\label{equivwlindelof} $K_\Aaa$ is monolithic if and only if $\Aaa$ is locally small
if and only if $C(K_\Aaa)$ is Lindel\"of in the weak topology.
\end{prop}
\begin{proof}
The first equivalence is immediate from the definitions.

If $K_\Aaa$ is not monolithic, then $C(K_\Aaa)$ is not Lindel\"of in the weak topology
by Lemma 2.2. of \cite{KoszZiel} which is based on Theorem 2 of \cite{pollindelof}.
The shortest way of proving the other implication is to use forcing method. Suppose $K_\Aaa$
is monolithic but $C(K_\Aaa)$ is not Lindel\"of in the weak topology. Let $\kappa=|\Aaa|$.
Let $\{V_\xi: \xi<\kappa\}$ be an open cover of $C(K_\Aaa)$ without a countable subcover.
and let $P$ be a forcing notion which collapses $\kappa$ to $\omega_1$ and is countably closed.
Note that $K_\Aaa$ is the same  in the generic extension and so $\{V_\xi: \xi<\kappa\}$ remains an open cover of $C(K_\Aaa)$ without a countable subcover because $P$ does not introduce any new countable sets.
Moreover $\Aaa$ remains locally small by the same argument and the absoluteness of intersections.
In the generic extension we can however apply  Theorem 2.5. of \cite{KoszZiel} which is based on
the results of \cite{pollindelof} whose generalization to higher cardinals seems at least tedious.
Applying the above theorem, we obtain a contradiction.
\end{proof}

The following fact follows straight from the definition:

\begin{prop}\label{pghorerfg}
Let a Banach space $E$ have the controlled SCP. Then the dual unit ball $\ubal{E^*}$ with the weak star topology is monolithic.
\end{prop}

\begin{pf}
Fix a countable set $B \subs \ubal{E^*}$ and choose a projection $\map PEE$ such that $PE$ is separable and $B \subs P^*E^*$.
Then $\cl B$ is contained in the second countable space $P^*E^*$ considered with the weak star topology.
\end{pf}

It is not known whether monolithicity of the dual unit ball (perhaps plus some extra property of the Banach space) implies the controlled SCP.
We prove such a result for $C(K_\Aaa)$ spaces.

\begin{lm}\label{Lleneuoe}
Let $\Aaa$ be an almost disjoint family of countable sets with $X = \bigcup \Aaa$ and let $S \subs X$ and $\Bee \subs \Aaa$ be countable, such that $\cl S \cap \el_\Aaa \subs \el_\Bee$.
Then there exists a countable set $N \subs X$ such that $S \subs N$, $B \subs^* N$ for every $B\in \Bee$ and $A \cap N$ is finite for every $A \in \Aaa \setminus \Bee$.
\end{lm}

\begin{pf}
Let $T = \bigcup \Bee$ and let $\Dee \subs \Aaa \setminus \Bee$ be such that $\cl T = T \cup \el_{\Bee \cup \Dee} \cup \sn \infty$.
Since $K_\Aaa$ is monolithic, the family $\Dee$ is necessarily countable (possibly finite or empty).
Write $\Bee = \ciag B$ and $\Dee = \ciag D$, where we agree that $D_n=\emptyset$ for $\ntr$ in case $\Dee = \emptyset$.
Let $C_n = B_n \setminus (D_0\cup \dots \cup D_n)$ and define
$$N = S \cup \bigcup_{\ntr} C_n.$$
We claim that $N$ is as required.

Indeed, if $A\in \Bee$ then $A = B_n$ for some $\ntr$ and therefore $A \subs^* C_n \subs N$.
If $A \in \Dee$ then, by assumption, $A \cap S$ is finite and $A \cap C_n = \emptyset$ for all but finitely many $\ntr$.
Furthermore, $A\cap C_n \subs A \cap B_n$ is finite for $\ntr$, therefore $A \cap N$ is finite.
Finally, if $A \in \Aaa \setminus (\Bee \cup \Dee)$ then $A \cap S$ is finite and $A \cap \bigcup_\ntr C_n \subs A \cap T$ is finite, because $\el_A \notin \cl T$.
\end{pf}

\begin{prop}\label{equivmonolithic}
Let $\Aaa$ be an almost disjoint family consisting of countable sets.
Then $C(K_\Aaa)$ has the controlled SCP  if and only if $K_\Aaa$ is monolithic.
\end{prop}

\begin{pf}
In view of Proposition~\ref{pghorerfg}, we need to prove the ``if" part only.
So fix an almost disjoint family $\Aaa$ and assume $X = \bigcup\Aaa$.

Recall that, given $f \in C(K_\Aaa)$, there exists a countable set $S_f\subs X$ such that $f$ is constant on $K_\Aaa \setminus \cl S_f$.
Every functional $\mu \in C(K_\Aaa)^*$ can be identified with a real-valued Borel measure whose support $\supp(\mu)$ is separable and therefore countable, because $K_\Aaa$ is monolithic.

Fix countable sets $A \subs C(K_\Aaa)$ and $B \subs C(K_\Aaa)^*$.
Let $S \subs X$ be a countable set such that $S_f \subs S$ for $f \in A$ and $\supp(\mu) \subs \cl S$ for every $\mu \in B$.
Since $K_\Aaa$ is monolithic, $\cl S$ is countable, therefore there exists a countable $\Cee \subs \Aaa$ such that $\cl S = S \cup \el_\Cee \cup \sn \infty$.

By Lemma~\ref{Lleneuoe}, there exists a countable set $N\subs X$ such that $S \subs N$ and $A \subs^* N$ for $A \in \Cee$ and $A \cap N$ is finite for every $A \in \Aaa \setminus \Cee$.
Let $V = \cl N$ and $W = \cl(X \setminus N)$.
Note that each of the sets $V$ and $W$ is closed and $V \cap W = \sn \infty$.
Define a retraction $\map r{K_\Aaa}V$ by collapsing the set $W$ to $\sn \infty$.
Define $\map P{C(K_\Aaa)}{C(K_\Aaa)}$ by
$$P f = (f \rest V) \cmp r.$$
In other words, $(P f) \rest V = f \rest V$ and $P f (x) = f(\infty)$ for $x\in W$.
Clearly, $P$ is a linear projection of norm one and its range consists of all functions $f\in C(K_\Aaa)$ that are constant on $W$.
Since $V$ is second countable, $\im P$ is separable.

Finally, we have that $A \subs \im P$ because $S_f \subs N$ for $f\in A$, and $B \subs \im P^*$ because $\supp(\mu) \subs V$ for $\mu \in B$.
\end{pf}

\section{The structure of locally small almost disjoint families of cardinality $\omega_1$}

The purpose of this section is to prove Theorem \ref{structural}. This is a structural result
where ladder system spaces play an important role, so we turn to them first.
Let $\kappa$ be an uncountable cardinal. An indexed family $\mathcal L=(L_\alpha)_{\alpha\in S}$ for $S\subseteq
 E_\kappa(\omega)=\{\alpha\in \kappa: \cf(\alpha)=\omega\}$
 is called a \emph{ladder system} on $S$ if each element
$L_\alpha$ consisting of ordinals in $\kappa\setminus E_\kappa(\omega)$ is of order type $\omega$, $L_\alpha\subseteq\alpha$ and $\sup(L_\alpha)=\alpha$.

It is clear that $\mathcal L = (L_\alpha)_{\alpha \in S}$ is an almost disjoint family of countable sets.
The space $K_{\mathcal L}$ is called a \emph{ladder system compact space}.
We shall call it \emph{non-trivial} if $S$ is stationary. 
This is explained by the following folklore fact:

\begin{prop}\label{stationaryladder} Suppose that $\mathcal L=(L_\alpha: \alpha\in S)$ is a ladder system on $S\subseteq\omega_1$
Then $K_{\mathcal L}$ is Eberlein if and only if $S$ is non-stationary.
\end{prop}
\begin{proof} We will use \ref{lemmadisjoint}. Suppose that $K_{\mathcal L}$ is disjointifiable. 
Let $F_\alpha$ for $\alpha\in S$ be such that $L_\alpha\setminus F_\alpha$s form a pairwise
disjoint family.  Let $f(\alpha) = \min(L_\alpha\setminus F_\alpha)$.
This is a regressive function defined on
a stationary set, so it is constant on a stationary set, contradicting the disjointness.

Now, if $S$ is non-stationary, let $C\subseteq\omega_1$ be a club set disjoint from $S$.
For $\alpha\in S$ let $f(\alpha)=\max\{C\cap\alpha\}<\alpha$, hence $F_\alpha= L_\alpha\setminus (f(\alpha)+1)$
is finite. Note that this means that $L_\alpha\setminus F_\alpha$ is disjoint from
$L_{\alpha'}\setminus F_{\alpha'}$ whenever $f(\alpha)\not=f(\alpha')$. But pre-images under $f$ of 
one element sets are countable, hence ${\mathcal L}$ is a union of countably any disjointifiable families, and so
disjointifiable by Lemma~\ref{lemmadisjoint}(4).

\end{proof}

In our formalism, the non-isolated points of $K_\Aaa$ where $\Aaa$ is a ladder system
are $\{l_A: A\in\Aaa\}$ however we can associate them with limit ordinals $\alpha\in \omega_1$, namely
$l_{A_\alpha}$ is associated with $\alpha$. Also $\infty$ is associated with $\omega_1$.
This way we obtain a standard form of the ladder system space whose set of points is $[0,\omega_1]$.
Ladder system spaces were used by R. Pol as the first example of a Lindel\"of space $C_p(K)$ for which $K$ is not Corson compact, see \cite{pollindelof}, they were also exploited in \cite{ciesielskipol} in the
context of function spaces. It is clear that if $\mathcal L$ is a
 ladder system, then $K_{\mathcal L}$ is monolithic, therefore its space of continuous functions 
has the controlled SCP and all the properties of Theorem \ref{monolithiclist}. We will show that every locally small almost disjoint family of cardinality
$\omega_1$ is equivalent to a family made of ladder systems and pairwise disjoint families.

\begin{proof}[Proof of Theorem\ref{structural}]
Let $\mathcal A$ be a family of countable sets of cardinality $\omega_1$. We 
may assume that $\bigcup\Aaa$ consists of successor countable ordinals, i.e., 
$\omega_1\setminus E_{\omega_1}(\omega)$.

Note that using the local smallness of
$\Aaa$ by the standard closure argument for every $\alpha<\omega_1$ there is a
limit ordinal $\alpha<\delta<\omega_1$ such that if $A\cap\xi$ is infinite for some
$\xi<\delta$, then $\sup(A)<\delta$. Let us call this property of $\delta$ as $\mathcal P$.
Note that by the above observation the set of ordinals with property $\mathcal P$ is closed and unbounded in
$\omega_1$. Let $\{\delta_\alpha:\alpha<\omega_1\}$ be an increasing (and necessarily continuous)
enumeration of this club set. For
$A\in \Aaa$ let $\eta(A)=\min\{\alpha: |A\cap \delta_\alpha|\ \hbox{\rm is infinite}\}$. 
Directly from property $\mathcal P$ we obtain the following:

\begin{lm}\label{struclemma1} $\sup(A)<\delta_{\eta(A)+1}$ for every $A\in \Aaa$.

\end{lm}

\begin{lm}\label{struclemma2}
The family
$$\Aaa_1=\{A\in \Aaa: \sup(A)<\delta_{\eta(A)}\}$$ is disjointifiable.
\end{lm}
\begin{pf}
If $\sup(A)<\delta_{\eta(A)}$, then by the continuity of
the sequence of $\delta_\alpha$s there is $\xi(A)$ such that $\xi(A)+1=\eta(A)$,
$A\cap\delta_{\xi(A)}$ is finite and $A\subseteq \delta_{\eta(A)}$. 
On the other hand, by local smallness of $\Aaa$ there are only countably many
elements $A$ of $\Aaa$ such that $\eta(A)$ is fixed.
$\Aaa_1'=\{A\setminus \delta_{\xi(A)}: A\in \Aaa_1\}$  is a union of uncountably many 
countable families each with its sum included in an interval of the form $[\delta_\alpha,\delta_{\alpha+1})$.
So it is a union of countably many pairwise disjoint families, and so it is disjointifiable by Lemma~\ref{lemmadisjoint}(4), and so
$\Aaa_1$ is disjointifiable.
\end{pf}

Define $\Aaa_2=\Aaa\setminus \Aaa_1$.

\begin{lm}\label{struclemma3} If $A\in \Aaa_2$, then the order type of $A\cap \delta_{\eta(A)}$ is $\omega$.
\end{lm}
\begin{proof}
If the order type is different, then there is $\xi<{\delta_{\eta(A)}}$ such that
$A\cap \xi$ is infinite. Then by property $\mathcal P$ of $\delta_{\eta(A)}$
we would have that $A\in \Aaa_1$.
\end{proof}

Let $S$ be the set of all $\delta_\alpha<\omega_1$ such that there exists
$A\in \Aaa_2$ with $\eta(A)=\alpha$.
For each $\delta_\alpha\in S$ let $\{A^{\delta_\alpha}_n: n\in\Nat\}$ be an enumeration (possibly with repetitions)
of all $A\in \Aaa_2$ with $\eta(A)=\alpha$; it exists because $\Aaa$ is locally small.
For $\alpha\in S$ let 
$$B^{\delta_\alpha}_n= A^{\delta_\alpha}_n\cap {\delta_\alpha}\ \ \hbox{and}
\ \ C^{\delta_\alpha}_n= A^{\delta_\alpha}_n\setminus {\delta_\alpha}.$$
It is clear that $\Bee_n=\{B^{\delta_\alpha}_n: \alpha\in S\}$ are ladder systems on $S$.
So it is enough to prove that each 
$$\Cee_n=\{C^{\delta_\alpha}_n: \alpha\in S\}$$
is disjointifiable. But for this note that it is actually a disjoint family, as
$$C^{\delta_\alpha}_n\subseteq [\delta_{\eta(A_n^{\delta_\alpha})}, \delta_{\eta(A_n^{\delta_\alpha})+1}),$$ 
which follows from Lemma \ref{struclemma1}. 
\end{proof}

It should be clear by \ref{lemmadisjoint}(3) that $\Aaa+_\phi\Bee$ is disjointifiable if
both $\Aaa$ and $\Bee$ are disjointifiable.
So \ref{lemmadisjoint}(4) implies that if $\Aaa$ of \ref{structural} is not disjointifiable, i.e.,
$K_\Aaa$ is not  Eberlein compact, then at least one of the ladder systems of \ref{structural}
is nontrivial, in our construction this translates to the stationarity of the set $S$
by Proposition \ref{stationaryladder}.

Note that $\Aaa+_\phi\Bee$  is not equivalent to a nontrivial ladder system $\mathcal L$  if
$\Aaa$ is pairwise disjoint, $\phi$ is a bijection and $\Bee$ is arbitrary, possibly a nontrivial
ladder system. To see this,
let $i$ be the Boolean isomorphism from the algebra generated by $\Aaa+_\phi\Bee$
and finite subsets of $\bigcup(\Aaa+_\phi\Bee)$ and the algebra generated by $\mathcal L$
and finite subsets of $\bigcup\mathcal L$.  As elements of $\Bee$ are infinite, for each $L\in \mathcal L$
there will
be distinct $x_L\in \bigcup\Bee$ such that $i(\{x_L\})\in L$. But this contradicts the Pressing Down Lemma
and injectivity of $i$.

We should note that assuming $\square_{\omega_1}$ there
are non-disjointifiable locally small almost disjoint families of countable sets $\Aaa$ of cardinality $\omega_2$
 such that $\Bee\subseteq\Aaa$ is disjointifiable for every $\Bee$ of cardinality $\omega_1$.
This is just a ladder system on a stationary $E\subseteq\{\alpha\in\omega_2: \cf(\alpha)=\omega\}$
such that $E\cap \alpha$ is non-stationary in any $\alpha\in\omega_2$ of cofinality $\omega_1$
(23.6 \cite{jech}). So the structure of bigger families may not depend on the  
subfamilies of cardinality $\omega_1$.

\begin{question}Describe the structure of locally small almost disjoint
families of cardinalities bigger than $\omega_1$.
\end{question}

\section{A space $C(K_\Aaa)$ with controlled and continuous SCP but without
a projectional skeleton}

\begin{df}
Let $L\subseteq K$ be compact Hausdorff spaces.
We define
$$C(K||L)=\{f\in C(K): f \ \hbox{is constant on}\  L\}.$$
\end{df}

The subject of this section is the 

\begin{pf}[Proof of Theorem \ref{nonprojskeleton}]

Let $\Cee = \sett{C_\al}{\al \in S}$ be a ladder system on $\omega_1$ with $S$ stationary.
We will prove that the Banach space $C(K_\Cee)$ is the union of a continuous chain of $2$-complemented separable subspaces, yet it has no projectional skeleton.
The second statement follows from Proposition \ref{eberleinandskeletons} and Proposition \ref{stationaryladder} above.
We shall check that canonical separable subspaces of $E := C(K_\Cee)$ are $2$-complemented.
By a canonical separable subspace we mean
$$E_\al = C(K_\Cee || L_\al),$$
where $\al < \omega_1$ and $L_\al = (\omega_1 \setminus \al) \cup \setof{\el_{C_\xi}}{\xi \in S \setminus \al} \cup \sn \infty$.
It is easy to check that $\sett{E_\al}{\al < \omega_1}$ is a continuous chain of separable spaces whose union is $E$.
Define
$$K_\al = \al \cup \setof{\el_{C_\xi}}{\xi \in S \cap (\al + 1)} \cup \sn \infty.$$
Note that $K_\al$ is the closure of $\al = [0,\al)$ in $K_\Cee$ and $K_\al \cap L_\al = \dn \infty {\el_{C_\al}}$ if $\al \in S$, while $K_\al \cap L_\al = \sn \infty$ for every $\al \in \omega_1 \setminus S$.
Let $\map {r_\al}{K_\Cee}{K_\Cee}$ be such that $r_\al \rest K_\al$ is the identity and either $r^{-1}(\infty) = L_\al\setminus \sn {\el_{C_\al}}$ or $r^{-1}(\infty) = L_\al$, depending on whether $\al \in S$ or $\al \notin S$.
Then $r_\al$ is a continuous retraction onto $K_\al$.
In case $\al \in S$, the point $\el_{C_\al}$ is isolated in $L_\al$, therefore $r_\al$ is indeed continuous.

Fix $\al < \omega_1$.
If $\al \notin S$ then $L_\al$ is a clopen set and $E_\al$ is $1$-complemented in $E$, which is witnessed by the formula
$$P f = f \cmp r_\al.$$
Now assume that $\al \in S$ and let $U = C_\al \cup \sn \infty$.
Define $\map Q{E}{E_\al}$ by setting
$$Q f = f \cmp r_\al + (f(\el_{C_\al}) - f(\infty)) \cdot \chi_U.$$
It is clear that $Q$ is a linear operator with $\norm Q \loe 2$.
Observe that $Q f (\el_{C_\al}) = f(\el_{C_\al})$ and $Q f (\infty) = f(\el_{C_\al})$.
Furthermore, $Q f$ is constant on $L_\al$, that is, $Q f \in E_\al$.
On the other hand, if $f \in E_\al$ then $f = f \cmp r_\al$ and $f(\infty) = f(\el_{C_\al})$, therefore $Qf = f$.
Hence $Q$ is a projection onto $E_\al$.
It follows that $E_\al$ is $2$-complemented in $E$.
This completes the proof.
\end{pf}

\section{A non-weakly Lindel\"of space $C(K_\Aaa)$ which has the SCP}

The subject of this section is the proof of theorem \ref{nonmono}. First we note that
there cannot be a space like in this theorem with only countably many isolated points:

\begin{prop}\label{uncomplemented} %[Example 2, \cite{johnsonlind}]
Let $\Aaa$ be an uncountable almost disjoint family such that $\bigcup \Aaa$ is countable.
Then $C(K_\Aaa)$ fails the SCP.
\end{prop}

\begin{pf}
The space $C(K_\Aaa)$ has a canonical isometric copy of $c_0$ which is not complemented by Whitley's argument \cite{whitley}, as in the case of $c_0 \subs \ell_\infty$ (see also \cite[Thm. 17.3]{KKLP} for a more general statement).
This fact, together with Sobczyk's theorem implies that $C(K_\Aaa)$ cannot have the SCP.
\end{pf}

\begin{lm}\label{eventconst}
Suppose $\mathcal A$ is an infinite almost disjoint family of subsets of a set $X$. Then every $f\in C(K_{\mathcal A})$ 
is eventually constant, that is, it belongs to  $C(K_{\mathcal A}|| \cl(X\setminus Y))$ for some countable $Y\subseteq X$.
\end{lm}
\begin{proof} Clearly, if $f$ is constant on $X\setminus Y$, then $f$ is constant on $\cl(X\setminus Y))$. So,  if the lemma is false, then 
 there is $\varepsilon>0$ and uncountably many pairs $x_\xi, y_\xi$ of elements of $X$ such that $|f(x_\xi)-f(y_\xi)|>\varepsilon$. So, as neighborhoods of
$\infty$ are complements of countable sets, every neighborhood of $\infty$ contains a pair as above. This contradicts the continuity of
$f$ at $\infty$.
\end{proof}

\begin{lm}\label{lemmainjection} Suppose that $\mathcal A$ is an infinite almost disjoint family of countable sets 
with $\bigcup\mathcal A=X$. Let $Y\subseteq X$ be countably infinite, co-infinite such that $X\setminus Y$ is not covered by
finitely many elements of $\mathcal A$ and let  $\mathcal B=\{A\in \mathcal A: A\subseteq^* Y\}$.
Suppose that there exists an injection $F:Y\rightarrow X\setminus Y$ satisfying
for each $A\in \mathcal A$:
$$\hbox{ If}\ A\not\in \mathcal B,\  \hbox{then}\  F[A\cap Y]=^*A\cap F[Y]. \leqno (*)$$

Then there is a retraction $r: K_{\mathcal A}\rightarrow \cl (X\setminus Y)$ onto $\cl (X\setminus Y)$ 
 and in particular 
the space $C(K_{\mathcal A}|| \cl(X\setminus Y))$ is complemented in $C(K_{\mathcal A})$.
\end{lm}
\begin{proof} 
First note that  $\infty\in \cl(X\setminus Y)$ and that
$$\cl (X\setminus Y)=(X\setminus Y)\cup \el_{\mathcal A\setminus \mathcal B}\cup\{\infty\}.$$

Define $r: K_{\mathcal A}\rightarrow \cl (X\setminus Y)$ considering several cases:
If $x\in Y$, then $r(x)=F(x)$;
 if $x=\el_B$ where $B\in \mathcal B$, then $r(x)=\infty$ and 
of course, $r(x)=x$ when $x\in \cl (X\setminus Y)$.

We need to check the continuity of $r$ at each $x\in K_{\mathcal A}$. Of course, we need to consider
only the non-isolated points from $\el_{\mathcal A}\cup\{\infty\}$.
First note that 
 $$\hbox{If}\  A\in \mathcal B,\  \hbox{then}\  F[A\cap Y]\cap A'\  \hbox{is finite for each}\  A'\in \mathcal A.\leqno (**)$$
This is because by ($*$) if $A'\in \mathcal A\setminus \mathcal B$, then almost disjointness of $A$ and $A'$ implies
the almost disjointness of $F[A]$ and $A'$ since $F$ is an injection. On the other hand if $A'\in\mathcal B$, then
$A'\subseteq^* Y$ which is disjoint from $F[A]$. Going to the continuity of $r$, we have several cases.

If $x=\el_B$ for some $B\in \mathcal B$, then the convergent sequence $B$ which converges to $\el_B$ is
sent by $r$ to a convergent sequence $F[B]$ which by ($**$) converges to $\infty$ which is $r(\el_B)$ giving us
the continuity of $r$ at $x$.

If $x=\el_A$ for some $A\in \mathcal A\setminus \mathcal B$, then the convergent sequence $A$ which converges to $\el_A$ is
sent by $r$ to a convergent sequence $F[A\cap B]\cup (A\setminus B)$ which by ($*$) 
is included in the convergent sequence $A\setminus B$ which converges to $\el_A=r(\el_A)$ giving us
the continuity of $r$ at $x$.

Finally, to prove the continuity of $r$ at $\infty$, it is enough to consider $Z\subseteq X$ such that
$\infty\in \cl Z$ and prove that $\infty\in \cl(r[Z])$. The general case can be reduced to two sub-cases $Z\subseteq X\setminus Y$
and $Z\subseteq Y$. The first one is trivial as $r$ is the identity on $X\setminus Y$. In the second case
if $\infty$ were not in the closure of  $r[Z]$, then without loss of generality
we could assume that  $r[Z]=F[Z]$ is included in some  $A\in 
\mathcal A\setminus \mathcal B$, by ($*$) this gives
$$F[Z]\subseteq A\cap F[Y]=^*F[A\cap Y],$$
which gives $Z\subseteq^* A\cap Y$ since $F$ is an injection
and this is impossible as $\infty\in \cl(Z)$.

Define the projection $P: C(K_{\mathcal A})\rightarrow C(K_{\mathcal A}|| \cl(X\setminus Y))$ by
$$P(f)=f-f\circ r+f(\infty).$$
As $r$ is constant on $X\setminus Y$, any function of the form $P(f)$ is constantly equal to
$P(\infty)$ on $\cl (X\setminus Y)$. If $f\in  C(K_{\mathcal A}|| \cl(X\setminus Y))$, then it is equal to
$f(\infty)$ on $\cl(X\setminus Y)$ (as $\infty\in \cl(X\setminus Y)$), so $P(f)=f-f(\infty)+f(\infty)=f$,
hence $P$ is the required projection. 
\end{proof}

Let us now describe the main difficulty in finding an almost disjoint family satisfying
\ref{nonmono} and the way we overcome it.
The most well known space of the form $K_{\mathcal A}$ which  is not monolithic is obtained from 
an uncountable almost disjoint family $\mathcal A$ of subsets of $\Nat$. However, it does not have the SCP.
The reason why $c(K_{\mathcal A}|| K_{\mathcal A}\setminus
\Nat)\sim c_0$
 is not included in any separable complemented  subspace of $C(K_{\mathcal A})$ is that in this case, it would be
complemented in  $C(K_{\mathcal A})$ by Sobczyk's theorem. The latter is impossible for the same reason that
$c_0$ is not complemented in $\ell_\infty$. 

Looking for a non-monolithic space of the form $K_{\mathcal A}(X)$ we note that,  for such an $\mathcal A$ there must be a countable
set $Y\subseteq X$ such that $\{A\cap Y: A\in \mathcal A\}$ is uncountable, hence the situation is quite
similar to that of an uncountable almost disjoint family of subsets of $\Nat$,  that is obtaining the SCP seems
impossible.
Lemma \ref{lemmainjection} provides a solution to this problem.
 Let $\mathcal A$ be an uncountable almost disjoint family
of subsets of $\Nat$. Let $\Nat'=\{n': n\in \Nat\}$ where $n'$s are distinct elements
not belonging to $\Nat$. Define $F:\Nat\rightarrow \Nat'$ by $F(n)=n'$ and define
$$\mathcal A'=\{A\cup F[A]: A\in \mathcal A\}.$$
Applying Lemma \ref{lemmainjection} for  $X=\Nat\cup \Nat'$,  $Y=\Nat$, $F$, $\mathcal A=\mathcal A'$ as above, we obtain that
 $C(K_{\mathcal A'}||K_{\mathcal A'}\setminus\Nat)$
is complemented in $C(K_{\mathcal A'})$ even though $C(K_{\mathcal A}||K_{\mathcal A}\setminus\Nat)$ was not complemented in
$C(K_{\mathcal A})$. This remark explains how a non-monolithic space $K_{\mathcal A}$
with an uncomplemented subspace $C(K_{\mathcal A}||K_{\mathcal A}\setminus\Nat)$  can be extended to
a bigger space $K_{\mathcal A'}$  where already $C(K_{\mathcal A'}||K_{\mathcal A'}\setminus\Nat)$ is complemented.
This is done by bifurcating the points of $\Nat$ and enlarging the sets of the family $\mathcal A$.
Of course  an important point which makes this enlargement useful here is that $C(K_{\mathcal A'}||K_{\mathcal A'}\setminus\Nat)$
and $C(K_{\mathcal A}||K_{\mathcal A}\setminus\Nat)$ are isometric in a canonical way via the restriction
of a continuous function from $K_{\mathcal A'}$ to $K_{\mathcal A}$.

\vskip 13pt

For the rest of this section fix an
 increasing enumeration  $(\lambda_\eta:\eta\in\omega_1)$ of all limit
countable ordinals, in particular $\lambda_0=0$. For every $\eta>0$ with
$\eta\in\omega_1$ we fix a bijection $F_\eta:[0,\lambda_\eta)\rightarrow [\lambda_\eta,\lambda_{\eta+1})$.
The main idea of the proof of \ref{nonmono} is to construct an almost disjoint family $\mathcal A=
\{A_\eta: 0<\eta<\omega_1\}$ of subsets of
$\omega_1$ such that $A_\eta\subseteq \lambda_\eta$ for each $0<\eta<\omega_1$ so that the  injections
$F_\eta:\lambda_\eta\rightarrow (\omega_1\setminus\lambda_\eta)$ satisfy  the hypothesis of Lemma \ref{lemmainjection}.
Then, some essential parts of the enlargements of the families $\{A\cap \lambda_\eta: A\in \mathcal A\}$ described in the
previous paragraph are already  included in $\mathcal A$. This gives that 
$C(K_{\mathcal A}||\cl(\omega_1\setminus\lambda_\eta))$ are complemented in $C(K_{\mathcal A})$ for each $0<\eta<\omega_1$.
As any separable subspace of $C(K_{\mathcal A})$ is included in $C(K_{\mathcal A}||\cl(\omega_1\setminus\lambda_\eta))$ for some $\eta<\omega_1$ by Lemma~\ref{eventconst},
we obtain the SCP. So we need the following:

\begin{lm}\label{adfamily} There is a family $(A_\eta:\eta<\omega_1)$ of
subsets of $\wp(\omega_1)$  such that:
\begin{enumerate}
\item $A_\eta\subseteq \lambda_\eta$ for each $\eta\in (0,\omega_1)$,
\item $A_\eta\cap [\lambda_\xi, \lambda_{\xi+1})$ is infinite
for all $\xi<\eta<\omega_1$,
\item $A_\eta\cap [\lambda_\xi, \lambda_{\xi+1})=^*F_\xi[A_\eta\cap \lambda_\xi]$
for all $\xi<\eta<\omega_1$,
\item $A_\eta\cap A_\xi$ is finite for all $\xi<\eta<\omega_1$.
\item No finite union of the sets of the form $A_\xi\cap[0,\omega)$ cover
a co-finite subset of $[0,\omega)$.
\end{enumerate}
\end{lm}
\begin{proof}
We put $A_0=\emptyset$, $A^1\subseteq[0,\omega)\subseteq [\lambda_0,\lambda_1)$
any infinite co-infinite subset of $[0,\omega)$ and suppose that $A_\alpha$s were
constructed so that (1) - (5) are satisfied. To construct $A_\eta=A^\eta_\eta$
we will construct inductively $A_\eta^\beta$s for $\beta\leq \eta$ so that the following
are satisfied:

\begin{enumerate}
\item[a)] $A_\eta^\beta\subseteq \lambda_\beta$ for each $\beta\leq \eta$,
\item[b)] $A_\eta^\beta\cap [\lambda_\xi,\lambda_{\xi+1})$ is infinite
for all $\xi<\beta\leq \eta$,
\item[c)] $A_\eta^\beta\cap [\lambda_\xi, \lambda_{\xi+1})=^*F_\xi[A_\eta^\beta\cap \lambda_\xi]$
for all $\xi<\beta\leq \eta$,
\item[d)] $A_\eta^\beta\cap A_\xi$ is finite for all $\xi<\eta$ and all $\beta\leq\eta$.
\item[e)] $A_\eta^\beta\cap \lambda_{\beta'}=^*A^{\beta'}_\eta$ for
all $\beta'\leq \beta\leq\eta$.
\end{enumerate}

We put $A_\eta^0=\emptyset$, 
Using the inductive assumption (5) find an infinite co-infinite $A^1_\eta\subseteq[0,\omega)\subseteq [\lambda_0,\lambda_1)$ which is almost disjoint from all $A_\xi\cap [0,\omega)$ for $\xi<\eta$ and so that (5) still
holds. To construct $A_\eta^{\beta+1}$ given an $A_\eta^\beta$ just
put 
$$A_\eta^{\beta+1}=A^\beta_\eta\cup F_{\lambda_\beta}[A^\beta_\eta].$$
This gives a) - c), e) immediately from the inductive hypothesis a) - c), e). The item d) follows from
the inductive hypothesis d) and from items (1), (3) for $\xi<\eta$. 

Only the limit stage is nontrivial. Given a limit $\beta\leq \eta$ choose a strictly
increasing sequence $(\beta_n)_{n\in \Nat}$ with $\beta_0=0$ and its supremum
equal to $\beta$. Enumerate $\{\xi: \xi<\eta\}$ as $\{\xi_n:n\in \Nat\}$. Define
$$A_\eta^\beta=\bigcup_{n\in \Nat\setminus\{0\}}[A_\eta^{\beta_n}\cap
[\lambda_{\beta_{n-1},\lambda_{\beta_n}})\setminus \bigcup_{i\leq n} A_{\xi_i}].$$
Note that the above definition implies that 
 $A_\eta^{\beta_n}\cap \lambda_{\beta_{n}}=^* A_\eta^\beta \cap \lambda_{\beta_{n}}$
by the inductive assumption e) since there are only finitely many integers less than $n$, hence
e) follows since $\beta_n$s are unbounded in $\beta$.
a) is immediate. c) follows from
the fact that 
$$A^\beta_\eta\cap [\lambda_\xi,\lambda_{\xi+1})=A_\eta^{\beta_{n+1}}\cap [\lambda_\xi,\lambda_{\xi+1})
=^* F_\xi[A_\eta^{\beta_{n+1}}\cap \lambda_\xi]=^* F_\xi[A_\eta^{\beta}\cap \lambda_\xi]$$
for $n\in\Nat$ such that $\xi\in [\beta_n,\beta_{n+1})$ by the inductive assumption c)
since in the definition of $A_\eta^\beta$ we subtract finite sets by the inductive assumption d)
and from the fact that $A_\eta^{\beta_{n+1}}\cap \lambda_{\beta_{n+1}}=^* A_\eta^\beta \cap\lambda_{\beta_{n+1}}$.
b) follows from c) and the fact that $A_\eta^1$ is infinite.
d) follows from the fact that we subtract $\bigcup_{i\leq n} A_{\xi_i}$ 
in the construction of $A^\beta_\eta$.

Now we see that $A_\eta=A^\eta_\eta$ satisfies (1) - (5).
\end{proof}

\begin{proof}[Proof of Theorem \ref{nonmono}]
Let $\mathcal A=\{A_\eta:\eta<\omega_1\}$ be as in Lemma \ref{adfamily}.
First note that $K_{\mathcal A}$ is not monolithic. This follows from the fact that $[0,\omega)=[\lambda_0, \lambda_1)$
has $\el_{\mathcal A}$ in its closure, and this is because $A\cap [\lambda_0, \lambda_1)$ is infinite for each $A\in \mathcal A$ which is 2) of Lemma~\ref{adfamily}.
By Lemma~\ref{lemmainjection} for every $0<\eta<\omega_1$ the space  $C(K_{\mathcal A}|| \cl(\omega_1\setminus \lambda_\eta))$ is complemented in $C(K_{\mathcal A})$.
This together with Lemma~\ref{eventconst} completes the proof of Theorem \ref{nonmono}.
\end{proof}

%\begin{question} Is there a weakly Lindel\"of Banach space without a separable complemented subspace? \end{question}

%\begin{question}Is there a Grothendieck $C(K)$ space which is weakly Lindel\"of? \end{question}

\end{document}